\documentclass[11pt, a4paper]{amsproc}

\usepackage{amsmath,amssymb,amsthm,url}

\newcommand{\C}{\mathbb{C}}

\newcommand{\F}{\mathbb{F}}
\newcommand{\SL}{\mathrm{SL}}

\newtheorem{prop}{Proposition}
\newtheorem{lem}{Lemma}
\newtheorem{lemma}[lem]{Lemma}
\newtheorem{cor}{Corollary}
\newtheorem{corollary}[cor]{Corollary}
\newtheorem{thm}{Theorem}
\newtheorem{theorem}[thm]{Theorem}

\theoremstyle{remark}
\newtheorem{remark}{Remark}
\theoremstyle{plain}

\title{Cubic maps from the group of order $3$}

\author{Vadim Alekseev}
\address{Vadim Alekseev, Fakult\"at f\"ur Mathematik, TU Dresden, Germany}
\email{vadim.alekseev@tu-dresden.de}

\author{Andreas Thom}
\address{Andreas Thom, Fakult\"at f\"ur Mathematik, TU Dresden, Germany}
\email{andreas.thom@tu-dresden.de}

\date{\today}

\begin{document}

\begin{abstract}
The purpose of this note is to classify unital cubic maps from the cyclic group of order $3$ into an arbitrary non-abelian group. We show that the universal group admitting a unital cubic map from the cyclic group of order $3$ is infinite, give a concrete presentation and provide an infinite representation of it in ${\rm PSL}_3(\mathbb C)$, whose image is an arithmetic lattice commensurable with ${\rm PSL}_3(\mathbb Z[\omega])$, where $\omega$ is a primitive cube root of unity. As a consequence we obtain the existence of finite nilpotent groups of arbitrarily large nilpotency class admitting a unital cubic map from $C_3$ whose image generates the group.
\end{abstract}

\maketitle

\section*{Introduction and basic definitions}

The study of polynomial maps between groups was initiated by A. Leibman, \cite{MR1608723, MR1910931}, motivated by applications in ergodic theory and additive combinatorics. Let $G,H$ be groups and let $\varphi:G\to H$ be a map.
For $k\in G$ we define the \emph{finite difference} (or discrete derivative in direction $k$) by
\begin{equation}\label{eq:finite-difference}
    (\Delta_k\varphi)(g):=\varphi(kg)\varphi(g)^{-1},\qquad g\in G.
\end{equation}
Following \cite{MR1608723}, we define polynomial maps by iterating finite differences.
We say that $\varphi$ is \emph{polynomial of degree $-1$} if it is identically~$1$.
For $d\ge -1$, we say that $\varphi$ is \emph{polynomial of degree $d+1$} if, for every $k\in G$,
the map $\Delta_k\varphi:G\to H$ is polynomial of degree $d$.
In particular, polynomial maps of degree $0$ are constant, and polynomial maps of degree $1$
are products of a homomorphism and a constant (cf.\ \cite[Lemma~2.1]{AJ}).
We call $\varphi$ \emph{unital} if $\varphi(1)=1$. Throughout, we will assume that all polynomial maps are unital, since any polynomial map can be unitalized by multiplying with a suitable constant. We will refer to polynomial maps of degree at most $2$ as \emph{quadratic} and of degree at most $3$ as \emph{cubic}.

In \cite{AJ} the authors introduced a new approach to polynomial maps between non-abelian groups, which allows for a more explicit description of the structure of quadratic maps. Given a group $G$ and $k \in \mathbb N$, there is always a universal group ${\rm Pol}_k(G)$ which admits a universal polynomial map $\varphi \colon G \to {\rm Pol}_k(G)$ of degree $k$; any polynomial map of degree $k$ from $G$ to $H$ arises as the composition of $\varphi$ and a uniquely determined homomorphism from ${\rm Pol}_k(G)$ to $H$. In the case of quadratic maps, the structure of the universal group can be described in concrete terms, while the structure of these groups for $k \geq 3$ is still somewhat mysterious even in the simplest cases. Apart from $G=C_2$, where ${\rm Pol}_k(C_2)$ is isomorphic to $C_{2^k}$ for all $k \geq 0$, and perfect groups, see \cite[Corollary 1.5]{AJ}, where the natural map ${\rm Pol}_k(G) \to G$ is an isomorphism for all $k \geq 1$, there is not a single computation of ${\rm Pol}_k(G)$ in the literature for $k \geq 3$.

\medskip

In this note we apply the results from \cite{AJ} to classify cubic maps from the cyclic group of order $3$.  While this is of course a natural next question, the computation is somewhat non-trivial and the result unexpected. Throughout this paper we let $C_3=\langle \sigma\mid \sigma^3=1\rangle$ and $\tau:=\sigma^2$. As mentioned above, the group ${\rm Pol}_3(C_3)$ is the universal group admitting a unital cubic map $\varphi \colon C_3 \to {\rm Pol}_3(C_3)$, i.e. any unital cubic map $\alpha \colon C_3 \to G$ to an arbitrary group $G$ factors through a (unique) homomorphism $\bar \alpha \colon {\rm Pol}_3(C_3)\to G$, i.e. $\alpha = \bar \alpha \circ \varphi$.
Existence of the universal group ${\rm Pol}_3(C_3)$ follows easily from the fact that the category of groups admits arbitrary products, and a first description in terms of generators and relations is similar to the one for ${\rm Pol}_2(G)$, see also \cite[Remark 2.5]{AJ}.

\medskip

The following is our first main result:

\begin{theorem} \label{main}
The group ${\rm Pol}_3(C_3)$ is isomorphic to the group
$$\langle a,b \mid (ba)^3, (ab^{-1}a)^3, [ba,ab^{-1}a] \rangle,$$ where $a = \varphi(\sigma), b=\varphi(\tau)$ for the universal unital cubic map $\varphi \colon C_3 \to {\rm Pol}_3(C_3)$. 

The elements $a,b$ have infinite order and the group ${\rm Pol}_3(C_3)$ contains a non-abelian free subgroup.
\end{theorem}

While the first claim is a matter of direct computation, the second claim is more subtle and relies on a computer-assisted search for an infinite representation of ${\rm Pol}_3(C_3)$ in ${\rm PSL}_3(\mathbb C)$. As a result, the unital map $\varphi \colon C_3 \to {\rm PSL}_3(\mathbb C)$ given by 
\[
\varphi(\sigma)=
\begin{pmatrix}
\omega r^2 & \frac{-2\omega-1}{3}r^2 & \frac{-\omega-2}{3}r^2\\
-r^2 & -\omega r^2 & \frac{-2\omega-1}{3}r^2\\
-(\omega+1)r^2 & r^2 & r^2
\end{pmatrix}
\quad \mbox{and} \quad
\varphi(\tau)=
\begin{pmatrix}
0 & 0 & \frac{\omega+2}{3}r\\
 r & 0 & (\omega+1)r\\
0 & r & \omega r
\end{pmatrix}
\]
with $\omega\in\C$ a primitive cube root of unity and $r$ a root of $r^3=1-\omega$, is cubic and leads to a representation of ${\rm Pol}_3(C_3)$ in ${\rm PSL}_3(\mathbb C)$, whose image is an arithmetic lattice commensurable with ${\rm PSL}_3(\mathbb Z[\omega])$. This arithmetic lattice can be described explicitly, see Theorem \ref{thm:arithmetic}, and the remaining claims in Theorem~\ref{main} follow from direct computation in this representation. There is another representation of ${\rm Pol}_3(C_3)$ in ${\rm SL}_3(\F_3((u)))$ with arithmetic image, which is described in Theorem \ref{thm:arithmetic2}.

\medskip

For a unital cubic map $\varphi:G\to H$, each first difference $\Delta_k\varphi$ is polynomial of degree~$2$.
We will unitalize these first differences as follows.
For $k\in G$ define
\begin{equation}\label{eq:beta-definition}
\beta_k(g):=\varphi(k)^{-1}(\Delta_k\varphi)(g)=\varphi(k)^{-1}\varphi(kg)\varphi(g)^{-1},\qquad g\in G.
\end{equation}

Let $\varphi \colon C_3 \to G$ be a unital cubic map. As a consequence $\beta_{\sigma}$ and $\beta_{\tau}$ are unital quadratic maps from $C_3$ to $G$ and can be further analyzed using the results from \cite{AJ}. Section \ref{sec:quadratic} contains an explicit description of the structure of $\mathrm{Pol}_2(C_3)$. This is then used in Section \ref{sec:cubic}, which is devoted to the classification of cubic maps from $C_3$ and the proof of Theorem~\ref{main}. 

Finally, Section \ref{sec:arithmetic} contains a description of the arithmetic lattice in the locally compact group ${\rm PSL}_3(\mathbb C)$ arising as the image of ${\rm Pol}_3(C_3)$ under the representation described above. This part makes use of Jambor's $L_3$--$U_3$-quotient algorithm \cite{Jam12} and relies on computer-assisted searches and computations which we verified independently using Magma, see the accompanying code.

\section{Quadratic maps and $\mathrm{Pol}_2(C_3)$}
\label{sec:quadratic}

The following proposition corrects an erroneous claim in \cite[Remark 4.5]{AJ}, where the group $\mathrm{Pol}_2(C_3)$ was claimed to be isomorphic to the Heisenberg group with coefficients in $\mathbb{Z}/3\mathbb{Z}$, see Remark \ref{rem:Heisenberg}. 

\begin{prop}\label{prop:Pol2C3-symmetric}
We have
\begin{equation}\label{eq:symmetric}
\mathrm{Pol}_2(C_3)\;\cong\;\Big\langle a,b\ \Big|\ a^9=1,\ b^9=1,\ bab^{-1}=a^4,\ aba^{-1}=b^4\Big\rangle,
\end{equation}
where $a=\phi(\sigma)$ and $b=\phi(\tau)$ for the universal unital quadratic map
$\phi:C_3\to \mathrm{Pol}_2(C_3)$. The group $\mathrm{Pol}_2(C_3)$ has order $27$ and is isomorphic to $C_9 \rtimes C_3$, where the generator of $C_3$ acts with multiplication by $4$ on $C_9$.
\end{prop}

\begin{proof}
By \cite[Theorem~1.3]{AJ}, there is an isomorphism
\[
\mathrm{Pol}_2(C_3)\;\cong\;(\omega(C_3)\otimes_{\mathbb Z} C_3)\rtimes_{\psi} C_3,
\]
where $C_3$ acts on $\omega(C_3)$ by left multiplication, acts trivially on $C_3=\mathrm{ab}(C_3)$,
the multiplication is
\[
(\xi,g)(\eta,h)=(\xi+g\eta+\psi(g,h),gh),
\]
and $\psi(g,h)=c(g)\otimes \bar h$ with $c(g)=g-1$ \cite[Theorem~1.3 and its proof]{AJ}. Moreover the universal map is
\[
\phi(g)=(0,g)\qquad(g\in C_3)
\]
\cite[proof of Theorem~1.3]{AJ}.

Set $V:=\omega(C_3)\otimes_{\mathbb Z} C_3$ (written additively).
Since $\omega(C_3)$ is free abelian of rank $2$ with basis
$\alpha:=\sigma-1$ and $\beta:=\tau-1$, we have $V\cong (\mathbb Z/3\mathbb Z)^2$ with basis
\[
e_1:=\alpha\otimes \bar\sigma,\qquad e_2:=\beta\otimes \bar\sigma.
\]
In particular, $|V|=9$, hence $|\mathrm{Pol}_2(C_3)|=|V||C_3|=27$.

Write $a:=\phi(\sigma)=(0,\sigma)$ and $b:=\phi(\tau)=(0,\tau)$.
Using $\psi(g,h)=c(g)\otimes \bar h$ and $\bar\tau=2\bar\sigma$, we record
\[
\psi(\sigma,\sigma)=e_1,\quad \psi(\tau,\sigma)=e_2,\quad
\psi(\sigma,\tau)=2e_1,\quad \psi(\tau,\tau)=2e_2.
\]
Then
\begin{align*}
a^2 &= (0,\sigma)^2=(\psi(\sigma,\sigma),\tau)=(e_1,\tau),\\
a^3 &= (e_1,\tau)(0,\sigma)=(e_1+\psi(\tau,\sigma),1)=(e_1+e_2,1),
\end{align*}
so $a^9=1$. Similarly,
\begin{align*}
b^2 &= (0,\tau)^2=(\psi(\tau,\tau),\sigma)=(2e_2,\sigma),\\
b^3 &= (2e_2,\sigma)(0,\tau)=(2e_2+\psi(\sigma,\tau),1)=(2e_1+2e_2,1),
\end{align*}
so $b^9=1$.

For inverses we use $(0,g)^{-1}=(-\psi(g^{-1},g),g^{-1})$ \cite[proof of Theorem~1.3]{AJ}.
Thus
\begin{align*}
b^{-1} &= (-\psi(\sigma,\tau),\sigma)=(-2e_1,\sigma)=(e_1,\sigma),\\
a^{-1} &= (-\psi(\tau,\sigma),\tau)=(-e_2,\tau)=(2e_2,\tau).
\end{align*}
Now compute
\begin{align*}
bab^{-1}
&=(0,\tau)(0,\sigma)(e_1,\sigma)\\
&=(\psi(\tau,\sigma),1)(e_1,\sigma)\\
&=(e_2,1)(e_1,\sigma)\\
&=(e_1+e_2,\sigma),
\end{align*}
while
\[
a^4=a^3a=(e_1+e_2,1)(0,\sigma)=(e_1+e_2,\sigma),
\]
hence $bab^{-1}=a^4$. Likewise,
\begin{align*}
aba^{-1}
&=(0,\sigma)(0,\tau)(2e_2,\tau)\\
&=(\psi(\sigma,\tau),1)(2e_2,\tau)\\
&=(2e_1+2e_2,\tau)=b^4.
\end{align*}
So \eqref{eq:symmetric} holds in $\mathrm{Pol}_2(C_3)$.

Let $\Gamma$ be the group given by the presentation \eqref{eq:symmetric}.
Since $\mathrm{Pol}_2(C_3)$ is generated by $\phi(\sigma)$ and $\phi(\tau)$
(by construction of $\mathrm{Pol}_2(G)$ on generators $\phi(g)$ \cite[\S2, proof of Theorem~1.2]{AJ}),
the assignment $a\mapsto \phi(\sigma)$, $b\mapsto \phi(\tau)$ defines a surjective homomorphism
$\Gamma\twoheadrightarrow \mathrm{Pol}_2(C_3)$.

It remains to bound $|\Gamma|$.
From $bab^{-1}=a^4$ and $a^9=1$ we get
\[
ba^3b^{-1}=(bab^{-1})^3=(a^4)^3=a^{12}=a^3,
\]
so $\langle a^3\rangle$ is central in $\Gamma$ and has size $\le 3$.
In the quotient $\Gamma/\langle a^3\rangle$ the relation $bab^{-1}=a^4$ becomes $bab^{-1}=a$,
so $a$ and $b$ commute there; moreover $(\bar a)^3=(\bar b)^3=1$ because $a^9=b^9=1$.
Hence $|\Gamma/\langle a^3\rangle|\le 3^2=9$, so $|\Gamma|\le 27$.

Since $\Gamma$ surjects onto $\mathrm{Pol}_2(C_3)$ and both have order $27$, the map is an isomorphism.
\end{proof}

\begin{remark} \label{rem:Heisenberg}
Note that while the Heisenberg group with coefficients in $\mathbb{Z}/3\mathbb{Z}$ is also a non-abelian $3$-group of order $27$, it is not isomorphic to $\mathrm{Pol}_2(C_3)$, since the former has exponent $3$ while the latter has exponent $9$. In standard notation, the Heisenberg group with coefficients in $\mathbb{Z}/3\mathbb{Z}$ is the extraspecial group $3^{1+2}_+$, while $\mathrm{Pol}_2(C_3)$ is the extraspecial group $3^{1+2}_-$. These are the only non-abelian groups of order $27$.
\end{remark}

\section{Classification of cubic maps from $C_3$}
\label{sec:cubic}

We will now apply the results of the previous section to classify cubic maps on $C_3$ and to prove Theorem~\ref{main}.
Let $\varphi \colon C_3 \to G$ be a unital cubic map into an arbitrary group $G$. We set $a:= \varphi(\sigma)$ and $b:=\varphi(\tau)$. Then
\[
\beta_\sigma(h):= \varphi(\sigma)^{-1}\varphi(\sigma h)\varphi(h)^{-1}
\quad \mbox{and} \quad
\beta_\tau(h):= \varphi(\tau)^{-1}\varphi(\tau h)\varphi(h)^{-1}
\]
are unital quadratic maps from $C_3$ to $G$. We compute those maps explicitly in terms of $a,b$ as follows:

\begin{align*}
\beta_{\sigma}(\sigma)
&= a^{-1}\varphi(\sigma^2)\varphi(\sigma)^{-1}=a^{-1}ba^{-1},\\
\beta_{\sigma}(\tau)
&= a^{-1}\varphi(\sigma\tau)\varphi(\tau)^{-1}=a^{-1}\varphi(1)\,b^{-1}=a^{-1}b^{-1},\\
\beta_{\tau}(\sigma)
&= b^{-1}\varphi(\tau\sigma)\varphi(\sigma)^{-1}=b^{-1}\varphi(1)\,a^{-1}=b^{-1}a^{-1},\\
\beta_{\tau}(\tau)
&= b^{-1}\varphi(\tau^2)\varphi(\tau)^{-1}=b^{-1}\varphi(\sigma)b^{-1}=b^{-1}ab^{-1}.
\end{align*}

It follows that the pairs $(a^{-1}ba^{-1},a^{-1}b^{-1})$ and $(b^{-1}a^{-1},b^{-1}ab^{-1})$ must satisfy the relations from Proposition \ref{prop:Pol2C3-symmetric}. In particular, we get the following relations:
\begin{subequations}\label{eq:quadratic-relations}
\begin{align}
(a^{-1}ba^{-1})^9 &= 1, \label{eq:quadratic-relations-1}\\
(a^{-1}b^{-1})^9 &= 1, \label{eq:quadratic-relations-2}\\
(a^{-1}b^{-1})\,(a^{-1}ba^{-1})\,(a^{-1}b^{-1})^{-1} &= (a^{-1}ba^{-1})^4, \label{eq:quadratic-relations-3}\\
(a^{-1}ba^{-1})\,(a^{-1}b^{-1})\,(a^{-1}ba^{-1})^{-1} &= (a^{-1}b^{-1})^4, \label{eq:quadratic-relations-4}\\
(b^{-1}a^{-1})^9 &= 1, \label{eq:quadratic-relations-5}\\
(b^{-1}ab^{-1})^9 &= 1, \label{eq:quadratic-relations-6}\\
(b^{-1}ab^{-1})\,(b^{-1}a^{-1})\,(b^{-1}ab^{-1})^{-1} &= (b^{-1}a^{-1})^4, \label{eq:quadratic-relations-7}\\
(b^{-1}a^{-1})\,(b^{-1}ab^{-1})\,(b^{-1}a^{-1})^{-1} &= (b^{-1}ab^{-1})^4. \label{eq:quadratic-relations-8}
\end{align}
\end{subequations}

Note that \eqref{eq:quadratic-relations-5} follows from \eqref{eq:quadratic-relations-2} by conjugation with $a$. We set
\[
z := a^{-1}b^{-1},\quad w := a^{-1}ba^{-1},\quad v := a^{-1}(b^{-1}ab^{-1})a = z w^{-1}.
\]
Rewriting the relations, we end up with:
\begin{align*}
&w^9\stackrel{\eqref{eq:quadratic-relations-1}}{=}1, \quad v^9\stackrel{\eqref{eq:quadratic-relations-6}}{=}1, \quad z^9\stackrel{\eqref{eq:quadratic-relations-2}}{=}1\\
 &z w z^{-1}\stackrel{\eqref{eq:quadratic-relations-3}}{=}w^4,\quad w z w^{-1}\stackrel{\eqref{eq:quadratic-relations-4}}{=}z^4,\\
&vz v^{-1}\stackrel{\eqref{eq:quadratic-relations-7}}{=}z^4, \quad  z v z^{-1}\stackrel{\eqref{eq:quadratic-relations-8}}{=} v^4,
\end{align*}
where we applied conjugation with $a^{-1}$ to the last two relations. Now, observe that $wzw^{-1}=z^4$ and $z^9=w^9=1$ imply $w^{-1}z w = z^7$, since $w^{-1}=w^8$ and $4^8 \equiv 7 \pmod{9}$.
We compute
$$z^4=vzv^{-1} = zw^{-1}zwz^{-1} = z z^7 z^{-1} = z^7$$
and obtain $z^3=1$. But this implies that $z$ commutes with $v,w$ and hence $w^3=v^3=1$. Thus, the elements $a,b$ necessarily satisfy the following relations:
$$(ba)^3, (ab^{-1}a)^3, [ba,ab^{-1}a].$$

We define $\Gamma:= \langle a,b\mid (ba)^3, (ab^{-1}a)^3, [ba,ab^{-1}a]\rangle.$ Our computation already implies that ${\rm Pol}_3(C_3)$ is a quotient of $\Gamma$, sending $a$ to $\varphi(\sigma)$ and $b$ to $\varphi(\tau)$. We now seek to show that the converse also holds.

\begin{lemma} The map $\varphi:C_3\to \Gamma$ given by $\varphi(\sigma)=a$ and $\varphi(\tau)=b$ is cubic. \end{lemma}
\begin{proof}
We have to check that $\Delta_k\varphi$ is quadratic for $k=\sigma$ and $k=\tau$. We only check the first case, the second is similar. We have to check that $\beta_\sigma$ is quadratic, i.e., that the relations from Proposition \ref{prop:Pol2C3-symmetric} are satisfied for the pair $(a^{-1}ba^{-1},a^{-1}b^{-1})$. But this follows easily from the relations in $\Gamma$ as we have already seen above.
\end{proof}

\begin{proof}[Proof of Theorem \ref{main}] By the universal property of ${\rm Pol}_3(C_3)$, the map $\varphi \colon C_3 \to \Gamma$ factors through a homomorphism ${\rm Pol}_3(C_3) \to \Gamma$. This map is clearly surjective and a section of the previously discussed surjection in the inverse direction. This implies that both groups are isomorphic identifying $a$ with $\varphi(\sigma)$ and $b$ with $\varphi(\tau)$.
\end{proof}

\begin{corollary}
The abelianization of ${\rm Pol}_3(C_3)$ is isomorphic to $C_9 \times C_3$. In particular, any cubic map from $C_3$ to an abelian group is uniquely determined by elements $a = \varphi(\sigma),b = \varphi(\tau)$ in $G$ satisfying $a^9=b^9=(ab)^3=1$.
\end{corollary}


\section{Two concrete representations of ${\rm Pol}_3(C_3)$ with arithmetic image}
\label{sec:arithmetic}

Throughout, we identify $\Gamma$ with ${\rm Pol}_3(C_3)$ via the isomorphism from Theorem \ref{main}. We also define $\Gamma^{\circ}$ to be the kernel of the tautological homomorphism $\mu \colon \Gamma \to C_3$ given by $\mu(a)=\sigma$ and $\mu(b)=\tau$.

\subsection{A representation in characteristic zero}

Fix a primitive cube root of unity $\omega\in\C$, so $\omega^2+\omega+1=0$. Let $r$ be a root of $r^3=1-\omega$ and set $E:=\mathbb Q(r)$. 
There is a canonical identification of $\mathbb Z[\omega]/((1-\omega))$ with $\mathbb F_3$.
We denote by ${\rm PSL}_3(\mathbb Z[\omega]; I)$ the congruence subgroup associated with the ideal $I$ and define
$K_n :={\rm PSL}_3(\mathbb Z[\omega]; ((1-\omega)^n))$ for $n=0,1,2$.
Note that $K_0={\rm PSL}_3(\mathbb Z[\omega])$ and $K_0/K_1= {\rm PSL}_3(\mathbb F_3), K_2/K_1 = {\mathfrak {sl}}_3(\mathbb F_3)$.
We also denote by ${\rm U}_+(\mathbb F_3) \subseteq {\rm PSL}_3(\mathbb F_3)$ the subgroup of upper triangular unipotent matrices.

\begin{theorem} \label{thm:arithmetic}
The assignment
\[
\pi(a)=
\begin{pmatrix}
\omega r^2 & \frac{-2\omega-1}{3}r^2 & \frac{-\omega-2}{3}r^2\\
-r^2 & -\omega r^2 & \frac{-2\omega-1}{3}r^2\\
-(\omega+1)r^2 & r^2 & r^2
\end{pmatrix},
\qquad
\pi(b)=
\begin{pmatrix}
0 & 0 & \frac{\omega+2}{3}r\\
 r & 0 & (\omega+1)r\\
0 & r & \omega r
\end{pmatrix}.
\]
extends to a representation of $\pi \colon \Gamma \to {\rm PSL}_3(E)$. The image of $\pi$ is commensurable with $K_0={\rm PSL}_3(\mathbb Z[\omega])$, hence arithmetic. More precisely,

\begin{enumerate}
    \item $\pi(\Gamma) \cap K_0 =\pi(\Gamma^{\circ})$ has index $3$ in $\pi(\Gamma)$,
    \item $\pi(\Gamma) \cap K_0$ has index $2^4 \cdot 3 \cdot 13 =624$ in $K_0$ and moreover
\begin{enumerate}
    \item $K_2 \subseteq \pi(\Gamma)$
    \item $(\pi(\Gamma) \cap K_1)/K_2= \left\{ \begin{pmatrix} 0 & a & b \\ c & d & e \\ 0 & f & -d  \end{pmatrix} \mid a,b,c,d,e,f \in \mathbb F_3 \right\} \subseteq {\mathfrak {sl}}_3(\mathbb F_3)$, and
    \item $(\pi(\Gamma) \cap K_0) / (\pi(\Gamma) \cap K_1) = {\rm U}_+(\mathbb F_3) \subseteq {\rm PSL}_3(\mathbb F_3)$.
\end{enumerate}
\end{enumerate}
\end{theorem}
\begin{proof}[Outline of the proof]
Once the matrices are given, it is a matter of direct computation to check that the relations in $\Gamma$ are satisfied. Generators of $\Gamma^\circ$ can be computed using the Nielsen-Schreier algorithm and it can be checked directly that $\pi(\Gamma^\circ)\subseteq K_0$ and that the quotient $\pi(\Gamma)/\pi(\Gamma^\circ)$ is isomorphic to $C_3$. This proves (1).

The rest of the claims follow from direct computation and standard facts about congruence subgroups. Indeed, we construct explicit words in the generators $a,b$ of $\Gamma$ that yield standard unipotent matrices with entries in $3 \mathbb Z[\omega]$ and $\langle U_-(3 \mathbb Z[\omega]), U_+( 3 \mathbb Z[\omega])\rangle={\rm PSL}_3(\mathbb Z[\omega]; 3 \mathbb Z[\omega])$, a standard fact since $3\mathbb Z[\omega]$ is a Euclidean domain, see \cite{bass}. This proves (2a). The claims in (2b) and (2c) are verified by direct computation.

All computations are certified by the Magma transcript \path{char0cert.m} and documented in \path{char0cert.txt}\end{proof}

\begin{corollary}
The group $\pi(\Gamma)$ is an arithmetic lattice in ${\rm PSL}_3(\mathbb C)$.
\end{corollary}

Polynomial maps to nilpotent groups are particularly important; see \cite{MR1910931}, where various structure results have been obtained. In \cite[Remark 4.2]{AJ} it was observed that the image of a quadratic map from $\mathbb Z$ can generate a finite nilpotent group of arbitrarily large nilpotency class. The following corollary shows that the same holds for cubic maps from $C_3$.

\begin{corollary}
There are finite nilpotent groups of arbitrarily large nilpotency class admitting a cubic map from $C_3$ whose image generates the group.
\end{corollary}
\begin{proof}
As a consequence of the preceding theorem, the group $\pi(\Gamma)$ is residually $3$-finite. Since finite $3$-groups are nilpotent and $\pi(\Gamma)$ is not itself nilpotent, there is no bound on the nilpotency class of finite quotients of $\pi(\Gamma)$. Since $\pi(\Gamma)$ is a quotient of $\Gamma$, the same holds for $\Gamma$. However, homomorphisms from $\Gamma$ correspond to cubic maps from $C_3$, so there are finite nilpotent groups of arbitrarily large nilpotency class admitting a cubic map from $C_3$ whose image generates the group.
\end{proof}

\subsection{A representation in characteristic $3$}
\label{sec:functionfield}

In addition to the complex representation from Section~\ref{sec:arithmetic}, there is also an explicit
representation of $\Gamma$ over a global function field of characteristic $3$. This allows to view $\Gamma^\circ$ as a lattice in $\SL_3(\F_3((u)))$.

Consider the function field $\F_3(t)$, set $u:=t^3$ and view $\F_3(u)$ as a subfield of $\F_3(t)$ via the
embedding $u\mapsto t^3$.

\begin{theorem}
\label{thm:arithmetic2}
The assignment
\[
\rho(a)=
\begin{pmatrix}
0 & 0 & 1/t\\
2/t & 0 & 0\\
2t^2 & 2t^2 & 0
\end{pmatrix},
\qquad
\rho(b)=
\begin{pmatrix}
0 & 2t & 0\\
0 & 0 & 2/t^2\\
t & 0 & t
\end{pmatrix}
\]
extends to a representation $\rho:\Gamma\to \SL_3(\F_3(t))$.
\begin{enumerate}
    \item We have $\rho(\Gamma^\circ) \leq \SL_3(\F_3[u])$.
    \item
    \begin{enumerate}
    \item $\rho(\Gamma^\circ)$ contains the principal congruence subgroup
    $\SL_3(\F_3[u],(u^2))$.
    In particular, $\rho(\Gamma^\circ)$ has finite index in $\SL_3(\F_3[u])$.
    \item The quotient
    \[
    \bigl(\rho(\Gamma^\circ) \cap \SL_3(\F_3[u]),(u)\bigr)\big/\bigl(\rho(\Gamma^\circ) \cap \SL_3(\F_3[[u]]),(u^2)\bigr)
    \]
    is isomorphic to a $7$-dimensional subspace of ${\mathfrak{sl}}_3(\F_3)$, and
    \item $\rho(\Gamma^\circ)/(\rho(\Gamma^\circ)\cap \SL_3(\F_3[u],(u)))$ is isomorphic to ${\rm U}_+(\F_3)$, the subgroup of upper triangular unipotent matrices in $\SL_3(\F_3)$. 
\end{enumerate}
\end{enumerate}
\end{theorem}

\begin{proof}[Outline of the proof]
Again, it is a matter of computing that $\rho(a),\rho(b)$ satisfy the relators of $\Gamma$, hence
define a representation $\rho$. It also computes Schreier generators for the index-$3$ kernel $H$ and matches
the images of $w_1,\dots,w_4$ under $\rho$ with explicit matrices $h_1,\dots,h_4\in\SL_3(\F_3[u])$.

Next, the certificate produces, root-by-root, explicit words in $h_1,\dots,h_4$ evaluating to the elementary
matrices $E_{ij}(u^2)$ and $E_{ij}(u^3)$ for all $i\neq j$. Using the Steinberg commutator relations in
$\SL_3(\F_3[u])$, this implies that $\rho(\Gamma^{\circ})$ contains all elementary matrices $E_{ij}(f(u))$ with $f(u)\in (u^2)$.
Since $\F_3[u]$ is a Euclidean domain \cite{bass}, one has
\begin{align*}
\SL_3(\F_3[u],(u^2)) &= E_3(\F_3[u],(u^2)),\\
\SL_3(\F_3[u],(u^2)) &\le \rho(\Gamma^\circ).
\end{align*}
Finally, $\SL_3(\F_3[u],(u^2))$ has finite index in $\SL_3(\F_3[u])$, because $\SL_3(\F_3[u]/(u^2))$ is finite.
All computations are certified by the Magma transcript \path{char3cert.m} and documented in \path{char3cert.txt}.
\end{proof}

As a corollary, we obtain a second arithmetic representation of $\Gamma^\circ$.

\begin{corollary} The group $\rho(\Gamma^\circ)$ is an arithmetic lattice in $\SL_3(\F_3((u)))$. \end{corollary}

\subsection{Final discussion}

Despite these attractive presentations of $\Gamma$, we could not identify the origin of these representations and determine if their existence is an isolated phenomenon or part of a larger picture. It seems we have struck an iceberg without having any sense of how large it is. In addition, basic questions about $\Gamma$ remain open, for example we do not know whether $\Gamma$ is residually finite or even linear. Unlike for ${\rm Pol}_2(G)$, we lack any conceptual description of ${\rm Pol}_3(C_3)$.

\medskip

Note however that we obtain the following peculiar consequence of Margulis seminal work on lattices in algebraic groups over local fields, see \cite{Margulis}. 

\begin{corollary}
The image of $\Gamma^\circ$ under the map $\pi \times \rho \colon \Gamma^\circ \to {\rm PSL}_3(\mathbb Z[\omega]) \times \SL_3(\F_3[u])$ contains a finite index subgroup.
\end{corollary}
\begin{proof}
Consider the image of the product map:
\[
\Delta:=(\pi\times\rho)(\Gamma^\circ)\le \pi(\Gamma^\circ)\times\rho(\Gamma^\circ).
\]
By Theorem~\ref{thm:arithmetic}, the subgroup $\pi(\Gamma^{\circ})$ has finite index in ${\rm PSL}_3(\mathbb Z[\omega])$.
By Theorem~\ref{thm:arithmetic2}, the subgroup $\rho(\Gamma^{\circ})$ has finite index in $\SL_3(\F_3[u])$.

Let $p_i$ for $i \in \{1,2\}$ be the coordinate projections. By definition of $\pi(\Gamma^{\circ})$ and $\rho(\Gamma^{\circ})$, both $p_1$ and $p_2$ are
surjective, hence $\Delta$ is a subdirect product.

Define $N_2:=\ker(p_1)$ and $N_1:=\ker(p_2)$. Then $N_2\lhd\rho(\Gamma^{\circ})$ and $N_1\lhd\pi(\Gamma^{\circ})$. By Margulis' Normal Subgroup Theorem \cite{Margulis}, each $N_i$ is either finite or has finite index.

We claim that neither $N_1$ nor $N_2$ can be finite. Assume that $N_2$ is finite.
Then $p_1$ induces an isomorphism $\Delta/N_2\cong \pi(\Gamma^{\circ})$.
Via this identification, the projection $p_2$ induces a surjection
\[
f:\pi(\Gamma^{\circ})\twoheadrightarrow \rho(\Gamma^{\circ})/N_2.
\]
Its kernel is a normal subgroup of $\pi(\Gamma^{\circ})$, hence finite or of finite index by Margulis.
Since $\rho(\Gamma^{\circ})/N_2$ is infinite, the kernel cannot have finite index and must be finite.

By Margulis' Superrigidity Theorem \cite{Margulis}, the map $f$ extends to a homomorphism ${\rm PSL}_3(\mathbb C)\to {\rm PSL}_3(\F_3((u)))$, which is absurd. We can argue similarly for $N_1$ in place of $N_2$.

Consequently, $N_1$ and $N_2$ have finite index in $\pi(\Gamma^{\circ})$ and $\rho(\Gamma^{\circ})$.
Since $(N_1\times\{1\})\le\Delta$ and $(\{1\}\times N_2)\le\Delta$, we get $N_1\times N_2\le \Delta$.
As $N_1\times N_2$ has finite index in $\pi(\Gamma^{\circ})\times\rho(\Gamma^{\circ})$, it follows that $\Delta$ contains a finite index subgroup of
${\rm PSL}_3(\mathbb Z[\omega]) \times \SL_3(\F_3[u])$.
\end{proof}

\begin{remark}
Both representation $\pi$ and $\rho$ give rise, by reduction modulo $1-\omega$ and $u$ respectively, to homomorphisms
\[
\bar \pi, \bar \rho \colon \Gamma^\circ \to {\rm PSL}_3(\mathbb F_3).
\]
The images of $\bar \pi$ and $\bar \rho$ are isomorphic to the upper triangular unipotent subgroup of ${\rm PSL}_3(\mathbb F_3)$, but the homomorphisms are not related by an automorphism. We found that surprising, since our first expectation was a compatibility between the two representations.
\end{remark}

\section*{Acknowledgements}

We used GPT-5.2 in order to improve on the existing implementation of the L3-U3-quotient algorithm in Magma and set up scripts for the necessary computer search, which was crucial for finding the infinite representations described in this paper. This has also revealed a mistake in the previous version of the code, which led to an incorrect output in some functions of the algorithm.

\end{document}